\documentclass[reqno,12pt]{amsart}

\usepackage[cp1251]{inputenc}
\usepackage[english]{babel}

\usepackage{amsmath,amsthm,amssymb}
\newtheorem{theorem}{Theorem}[section]
\newtheorem{lemma}[theorem]{Lemma}
\newtheorem{corollary}[theorem]{Corollary}
\newtheorem{remark}[theorem]{Remark}
\newtheorem{example}[theorem]{Example}
\newtheorem{proposition}[theorem]{Proposition}
\theoremstyle{definition}

\numberwithin{equation}{section}

\begin{document}

\title[]
{Regularity of solutions to Kolmogorov equations with perturbed drifts}

\author{Vladimir I. Bogachev$^{1}$,
Egor D. Kosov$^{2}$,
 Alexander V. Shaposhnikov$^{3}$
 }

\thanks{$^1$Department of Mechanics and Mathematics,
Moscow State University, Moscow, 119991, Russia;
National Research University, Higher School of Economics,  Usacheva 6, 119048 Moscow, Russia,
vibogach@mail.ru}
\thanks{$^2$Department of Mechanics and Mathematics,
Moscow State University, Moscow, 119991, Russia;
National Research University, Higher School of Economics,  Usacheva 6, 119048 Moscow, Russia,
ked${}_{-}$2006@mail.ru}
\thanks{$^3$Steklov Mathematical Institute of Russian Academy of Science,
8 Gubkina Street, 119991 Moscow, Russia, shal1t7@mail.ru, ashaposhnikov@mi.ras.ru}

\begin{abstract}
We prove that a probability solution of the stationary Kolmogorov equation
generated by a  first order perturbation $v$ of the Ornstein--Uhlenbeck
operator $L$ possesses a highly integrable density with respect to the Gaussian
measure satisfying the non-perturbed equation provided that $v$ is sufficiently
integrable. More generally, a similar estimate is proved for solutions to
inequalities connected with Markov semigroup generators under the curvature condition $CD(\theta,\infty)$.
For perturbations from $L^p$ an analog of the Log-Sobolev inequality is obtained.
It is also proved in the Gaussian case that the gradient of the density
is integrable to all powers. We obtain dimension-free bounds on the density and its gradient,
which also  covers the infinite-dimensional case.


\end{abstract}

\maketitle

\section{Introduction}

There is a vast literature on perturbations of Ornstein--Uhlenbeck operators
and their generalizations by vector fields $v$ that can be regarded as small
as compared to the main first order term $-x$.
Many authors studied Dirichlet forms obtained by such perturbations,
diffusion processes with perturbed drifts, and the corresponding stationary distributions
or solutions to the stationary Fokker--Planck--Kolmogorov equation, both
in finite and infinite dimensions.
In particular,
Shigekawa \cite{Shig87} proved  his celebrated result
that, given a centered
Gaussian measure $\gamma$ with the Cameron--Martin
space $H$
on a separable Banach space~$X$, for every bounded Borel $H$-valued vector field $v$
there is a Borel probability  measure   $\mu$ on $X$ absolutely continuous with respect
to $\gamma$ and
satisfying the perturbed  Fokker--Planck--Kolmogorov equation
\begin{equation}\label{FPKeq1}
L_v^*\mu=0,
\end{equation}
where $L$ is the Ornstein--Uhlenbeck operator and
$$
L_v \varphi=L\varphi +(v,\nabla \varphi)_H.
$$
Moreover, $\mu=f\cdot\gamma$ with $f$ belonging to the Sobolev space
$W^{2,1}(\gamma)$ over~$\gamma$.
In the finite-dimensional case
$$
L_v \varphi(x)=\Delta\varphi(x) -\langle x, \nabla \varphi(x)\rangle
+\langle v(x), \nabla \varphi(x)\rangle .
$$
The uniqueness of invariant measures was established in \cite{Shig87} in the case of measures
absolutely continuous with respect to~$\gamma$.
It was later shown in \cite{BR95} (see also  \cite{BRW01})
that all probability solutions to this equation are absolutely continuous with respect to $\gamma$
provided that $|v|_H\in L^2(\mu)$. Moreover, it is also true that for $f:=d\mu/d\gamma$
one has $\sqrt{f}\in W^{2,1}$ (but not always $f\in  W^{2,1}$) and
$$
\int \frac{|\nabla f|_H^2}{f^2}\, d\mu\le \int |v|_H^2 \, d\mu .
$$
In particular, the uniqueness result holds in the class of all probability solutions
if $v$ is bounded. For unbounded $b$ the uniqueness assertion can fail
even in the one-dimensional case.
The existence and uniqueness statements were reinforced by Hino~\cite{Hino98}
by replacing the boundedness  with the condition $\exp (\theta |v|_H^2)\in L^1(\gamma)$,
$\theta>2$.
In this case, one also  has the inclusions $f\in W^{2,1}(\gamma)$ and $f\in L^p(\gamma)$
for all $p\in [1, (\theta/2)^{1/2})$. Further extensions to measure metric spaces have
been obtained by Suzuki~\cite{Suzuki}.

For a general Fokker--Plank--Kolmogorov equation
$$
\Delta \mu -{\rm div}(b\mu)=0
$$
with respect to a Borel probability  measure $\mu$ on $\mathbb{R}^d$ with a Borel
vector field $b$ such that $|b|$ is locally integrable with respect to $\mu$,
which is understood as the identity
$$
\int [\Delta\varphi +\langle b,\nabla \varphi\rangle]\, d\mu=0
\quad \forall\, \varphi\in C_0^\infty,
$$
i.e., the equality is interpreted in the sense of distributions,
the following is known (see \cite{BR95},  \cite{BKR01}, and
surveys in \cite{BKR09}, \cite{BKRS}):

$\bullet$ the measure $\mu$ is always absolutely continuous with respect to Lebesgue measure,
i.e., is given by a density~$\varrho$;

$\bullet$ if $|b|^p$ is locally Lebesgue integrable or locally $\mu$-integrable with some
$p>d$, then the density $\varrho$ of $\mu$ belongs to the local Sobolev class
$W^{p,1}_{loc}$; if $|b|\in L^p(\mu)$ with some $p>d$, then
$\varrho \in W^{p,1}(\mathbb{R}^d)$ and $\|\varrho\|_\infty<\infty$;

$\bullet$ if $|b|\in L^2(\mu)$, then $\sqrt{\varrho}\in W^{2,1}(\mathbb{R}^d)$ and
$$
\int \frac{|\nabla \varrho|^2}{\varrho}\, dx\le \int |b|^2\, d\mu,
$$
where we set $|\nabla \varrho|^2/\varrho=0$ on $\{\varrho=0\}$.

It is not known whether there is an exact $L^p$-analog of the last result.
However, it has been shown in \cite{BPS19} that

$\bullet$  the condition $|b|\in L^1(\mu)$ is not
sufficient for the inclusion $|\nabla \varrho|\in L^1(\mathbb{R}^d)$, but gives
the inclusion $\varrho\in H^{r,\alpha}(\mathbb{R}^d)$ to the fractional Sobolev class
for each $r > 1$ and $\alpha < 1 - d(r - 1)/r$, where
$1 - d(r - 1)/r > 0$ whenever $1 < r < d/(d - 1)$; in particular, $\varrho\in L^s(\mathbb{R}^d)$ for each
exponent $s \in [1, d/(d - 1))$;

$\bullet$  if $|b|\in L^p(\mu)$ with some $p \in (1, d]$, then $\varrho\in W^{q,1}(\mathbb{R}^d)$
 for each exponent $q < d/(d + 1 - p)$, hence $\varrho\in  L^s(\mathbb{R}^d)$ for all $s < d/(d - p)$.

Although the condition $|b|\in L^1(\mu)$ is not enough for the membership of
$\varrho$ in the Sobolev class $W^{1,1}(\mathbb{R}^d)$, it implies
some weaker version of the logarithmic Sobolev inequality (see \cite{BSS19}).
Sufficient conditions for the boundedness of $\varrho$ can be found
in \cite{FFMP}, \cite{BKR06}, \cite{BRS07} along with some
other bounds (see \cite{BKRS} for a survey). Other related questions have been studied
in \cite{BDPRS06}, \cite{Es},
\cite{Hino00}, \cite{HinoM18},
\cite{Manca}, \cite{MetaPru05}, \cite{MetaS12},
for the infinite-dimensional case see \cite{DP}, \cite{DPZ}, and
\cite{Shig18}.

In this paper, we complement this picture by the following three results:

(i) a higher integrability result for the density of the solution to the perturbed equation,
in particular, some exponential integrability in the case of bounded $v$;
the main result (Theorem~\ref{main1}) deals with first order perturbations
of Markov generators satisfying the $CD(\theta, \infty)$ condition, so the aforementioned
Gaussian case is a quite special example, and in this case
we obtain that $\exp (\varepsilon |\log (f\wedge 1) |^2)\in L^1(\gamma)$ for all numbers
$\varepsilon < \bigl(4\pi^2\|\,|v|_H\|_\infty^2\bigr)^{-1}$, moreover, our estimates of the Orlicz norms
of the density are dimension-free;

(ii) if $|v|\in L^p(\mu)$ with some $p>2$, then $f \log^{\alpha}(1 + f)\in L^1(\mu)$ whenever
$\alpha < 2 \wedge \frac{p + 2}{4}$ (Theorem~\ref{th:main_l_p});

(iii) in the Gaussian case,
it is shown that
the density of the solution with respect to the Gaussian measure belongs
to the Gaussian Sobolev class $W^{p,1}(\gamma)$ for all $p\ge 1$
provided that the perturbation $v$ belongs to  a suitable Orlicz class,
e.g., is bounded, and a sufficient Orlicz integrability condition
is given for the inclusion in $W^{p,1}(\gamma)$ with a given~$p$,
moreover, our estimates of the $L^p$-norms of the gradient are dimension-free (Theorem~\ref{t-grad}).

In particular, the cited results of Shigekava and Hino are reinforced
and the first sufficient condition for the inclusion of the density in all
$W^{p,1}(\gamma)$ is given in terms of the integrability of~$v$; this result is new
even for bounded~$v$.
Some further extensions are possible to the case of a non-constant diffusion
term, but to keep the presentation less technical we confine ourselves to the unit diffusion matrix.

\section{Integrability of densities}

In this section, we work in the framework of abstract Markov triples (see \cite{BGL}).
Let  $(E, \mathcal{E}, \mu)$ be a probability space and
let $\{T_t\}_{t\ge0}$ be a strongly continuous semigroup of Markov operators on $L^2(\mu)$,
i.e., $0\le T_t \varphi\le 1$ whenever $0\le \varphi\le 1$ and $T_t1=1$.
We also assume that
$$
T_t\varphi\to\int\varphi\, d\mu\quad \hbox{as $t\to \infty$.}
$$
Let $L$ be the generator of this semigroup and
let $\mathcal{D}(L)$ be its domain.
We assume that  $L$ is symmetric:
$$
\int_E \varphi L\psi\, d\mu = \int_E \psi L\varphi\, d\mu
\quad \forall\, \varphi, \psi\in \mathcal{D}(L).
$$
Suppose, in addition, that $\mathcal{A}$ is an algebra of bounded measurable functions
dense in $\mathcal{D}(L)$ and in all $L^p(\mu)$ with $p<\infty$,
contains~$1$, and is stable under $L$ and  under compositions with $C^\infty$ functions of several variables,
which means that $F(f_1,\ldots,f_n)\in \mathcal{A}$ whenever $F\in C^\infty(\mathbb{R}^n)$
and $f_i\in \mathcal{A}$.
Let
$$
\Gamma(\varphi,\psi) = \frac{1}{2}[L(\varphi\psi) - \varphi L\psi - \psi L\varphi]
$$
for $\varphi, \psi \in \mathcal{A}$ and let
$$\Gamma(\varphi) = \Gamma(\varphi, \varphi).
$$
One has the following integration by parts formula:
$$
\int_E \varphi L\psi\, d\mu = - \int_E \Gamma(\varphi, \psi)\, d\mu,
\quad \varphi, \psi \in \mathcal{A}.
$$
Finally, we also assume that $L$ is a diffusion operator:
 for all $\Psi\in C^\infty(\mathbb{R}^k)$ and $\psi_1, \ldots, \psi_k\in \mathcal{A}$ one has
$$
L\Psi(\psi_1, \ldots, \psi_k) = \sum_{j=1}^{k}\partial_j\Psi(\psi_1,\ldots, \psi_k) L\psi_j
+\sum_{i, j=1}^{k}\partial^2_{i, j}\Psi(\psi_1, \ldots, \psi_k)\Gamma(\psi_i, \psi_j).
$$

We say that the curvature-dimension condition $CD(\theta, \infty)$ with some number $\theta$
holds if
$$
\Gamma_2(\varphi)\ge \theta \Gamma(\varphi)
\quad \forall\, \varphi\in \mathcal{A},
$$
where
$$
\Gamma_2(\varphi, \psi): =
\frac{1}{2}[L\Gamma(\varphi,\psi) - \Gamma(\varphi, L\psi) - \Gamma(L\varphi, \psi)],
\quad
\Gamma_2(\varphi) := \Gamma_2(\varphi, \varphi).
$$

It follows from our assumptions that
$$
L^*\mu=0,
$$
provided we define this equation as the identity
$$
\int L\varphi\, d\mu=0 \quad \forall\, \varphi\in \mathcal{A}.
$$
However, our main object will not be the solution $\mu$, but  probability solutions
$f\cdot \mu$ to some associated inequalities and equations $L_v^*(f\cdot\mu)=0$ with certain perturbations of $L$.

An important  example
is the case where  $\mu = e^{-W}\, dx$ for some $W\in C^2(\mathbb{R}^d)$ and
$$
L\varphi = \Delta\varphi - \langle\nabla W, \nabla\varphi\rangle.
$$
If $W(x) = |x|^2/2+C$, then $\mu$ is the standard Gaussian measure
on $\mathbb{R}^d$ denoted by $\gamma$.
One can also consider the Gaussian measure $\gamma$
in the infinite dimensional setting. In this case $L$
is the Ornstein--Uhlenbeck operator associated with~$\gamma$.
For the above measure the curvature-dimension condition $CD(\theta, \infty)$
holds if $D^2W\ge \theta {\rm I}$, since
$$
\Gamma_2(\varphi) = \|D^2\varphi\|_{HS}^2 + \langle D^2W\cdot\nabla\varphi, \nabla\varphi\rangle,
$$
where $\|\,\cdot\,\|_{HS}$ denotes the Hilbert--Schmidt norm.

\subsection{Orlicz-integrable drifts}

Let $m\ge1$ and let $\psi_m(t): = e^{t^m}-1$ for $t>0$.
Recall the definition of the Orlicz norm.
Let $\mu$ be a probability measure
on some space $E$. We say that
a measurable function $w$ belongs to the Orlicz class $L_{\psi_m}(\mu)$
if
$$
\int_E \psi_m(\lambda^{-1}|w|)\, d\mu<\infty \quad \hbox{for some } \lambda>0.
$$
We denote by $\|w\|_{L_{\psi_m}(\mu)}$ the Orlicz norm of $w$ defined by
$$
\|w\|_{L_{\psi_m}(\mu)}:=\inf\biggl\{\lambda\colon \int_E \psi_m(\lambda^{-1}|w|)\, d\mu\le1\biggr\}.
$$

Here is our first main result.

\begin{theorem}\label{main1}
Assume that the Markov triple $(E, \mu, \Gamma)$ satisfies  the
curvature-dimension condition $CD(\theta, \infty)$  with some $\theta>0$.
Suppose  that we are given a  probability density $f$ with respect to $\mu$ and
a measurable function  $w\ge 0$  such that
$\lambda:=\|w\|_{L_{\psi_m}(f\cdot\mu)}<\infty$ for some $m\in[2, +\infty)$.
Let also
$$
\int_E L\varphi f\, d\mu \le \int_E  w\sqrt{\Gamma(\varphi)} f\, d\mu
\quad \forall\, \varphi\in \mathcal{A}.
$$
Set $\mu_f:=f\cdot \mu$. The following assertions are true.

{\rm (i)} If $m=2$, then
$\mu_f(A)\le e^2[\mu(A)]^{\sigma_2}$,
where $\sigma_2:=\exp\bigl(-\frac{2\pi}{\sqrt{\theta}}\lambda\bigr)$.
In particular,
$$
\mu(f\ge t)\le e^2 t^{-\frac{1}{1-\sigma_2}}
$$
and $f\in L^p(\mu)$ for all $p<\frac{1}{1-\sigma_2}${\rm;}

{\rm (ii)} If $m>2$, then
$$
\mu(f\ge t)\le
e^{2} \exp\bigl(- \sigma_m[\ln t]^{\frac{2}{1+2/m}}\bigr) \quad \forall \, t>1
$$
and
$e^{\varepsilon[\ln \max\{f,1\}]^\frac{2}{1+2/m}}\in L^1(\mu)$ for all $\varepsilon<\sigma_m$,
where
$$
\sigma_m :=\frac{1-2/m}{1+2/m}
\Bigl(\frac{2\pi}{\sqrt{\theta}}\lambda(1-2/m)\Bigr)^{-\frac{2}{1+2/m}}.
$$

{\rm (iii)} If $w$ is bounded, then
$$
\mu(f\ge t)\le
e^{2} e^{-\sigma_\infty[\ln t]^2}\quad \forall\, t>1
$$
and
$e^{\varepsilon[\ln \max\{f,1\}]^2}\in L^1(\mu)$ for all $\varepsilon<\sigma_\infty$,
where
$\sigma_\infty := \bigl(\tfrac{2\pi}{\sqrt{\theta}}\|w\|_\infty\bigr)^{-2}$.
\end{theorem}
\begin{proof}
Recall that for every positive function
$\varphi\in \mathcal{A}$ one has (see \cite[Theorem 4.7.2 and Theorem 5.5.5]{BGL})
$$
T_t\varphi^2 - (T_t\varphi)^2\ge\tfrac{e^{2\theta t}-1}{\theta}\, \Gamma(T_t\varphi),\quad
\bigl(T_s(T_{t-s}\varphi)^q\bigr)^{1/q}\le (T_t\varphi^p)^{1/p},
$$
whenever $\frac{q-1}{p-1}=\frac{e^{2\theta t}-1}{e^{2\theta s}-1}$.

Fix $\alpha\in(0,1)$. Let $\varphi\in \mathcal{A}$ be a function such that $\alpha\le\varphi\le e^{-2}<1$.
Let us consider the function
$$
F(t)=\int_E T_t\varphi f\, d\mu.
$$
Note that $\alpha\le F(t)\le e^{-2}$ and that for every $s\in(0, t)$ we have
\begin{multline*}
F'(t) = \int_E LT_t\varphi f\, d\mu
\le
\int_E w \sqrt{\Gamma(\varphi)} f\, d\mu
=
\int_E w \sqrt{\Gamma(T_sT_{t-s}\varphi)} f\, d\mu
\\
\le
\tfrac{\sqrt\theta}{\sqrt{e^{2\theta s}-1}}
\int_E w\bigl(T_s(T_{t-s}\varphi)^2\bigr)^{1/2}f\, d\mu
\le
\tfrac{\sqrt\theta}{\sqrt{e^{2\theta s}-1}}
\int_E w (T_t\varphi^{r/(r-1)})^{1-1/r}f\, d\mu
\\
=
\tfrac{\sqrt{\theta}}{\sqrt{e^{2\theta t}-1}} \sqrt{r-1}
\int_E w (T_t\varphi^{r/(r-1)})^{1-1/r}f\, d\mu
\\
\le
\tfrac{\sqrt{\theta}}{\sqrt{e^{2\theta t}-1}}\sqrt{r-1}
\Bigl(\int_E w^rf\, d\mu\Bigr)^{1/r}[F(t)]^{1-1/r},
\end{multline*}
where $r-1=\frac{e^{2\theta t}-1}{e^{2\theta s}-1}$.
We note that for all $a>0$ one has
$a^xe^{-a}\le x^x$, since this is true for $x\ge a$ and for $x\le a$
one has $x\log a\le a+x\log x$ by differentiation in~$a$.
This implies the bound
$$
\biggl(\int_E w^rf\, d\mu\biggr)^{1/r}
\le
\biggl(\lambda^{r}\bigl(\frac{r}{m}\bigr)^{r/m}\int_E e^{\lambda^{-m}w^m}f\, d\mu\biggr)^{1/r}
\le 2^{1/r}\lambda r^{1/m}
$$
by taking $a=\lambda^{-m}w^m$ and $x=r/m$.
Thus, for all $r\ge2$ we have
$$
|F'(t)|
\le
\sqrt{2}\lambda\tfrac{\sqrt{\theta}}{\sqrt{e^{2\theta t}-1}} r^{1/2+1/m}[F(t)]^{1-1/r}.
$$
We now take $r= -\ln F(t)$ and, since $\sqrt2 e\le 4$, obtain the bound
$$
|F'(t)|
\le
4\lambda\tfrac{\sqrt{\theta}}{\sqrt{e^{2\theta t}-1}}[-\ln F(t)]^{1/2 + 1/m} F(t).
$$
First we consider the case $m=2$.
In this case, for all $a,b\in(0,+\infty)$, one has
$$
\ln[-\ln F(b)] - \ln[-\ln F(a)]
=
\int_a^b \frac{-F'(t)}{[-\ln F(t)] F(t)}\, dt
\le
4\lambda\int_a^b\frac{\sqrt{\theta}}{\sqrt{e^{2\theta t}-1}}\, dt
\le \tfrac{2\pi}{\sqrt{\theta}}\lambda.
$$
Thus,
$$
F(a) \le [F(b)]^{\sigma_2},
$$
where $\sigma_2 = \exp\bigl(-\frac{2\pi}{\sqrt{\theta}}\lambda\bigr)$.
Taking the limit as $a\to0$ and $b\to+\infty$, we obtain
$$
\int_E\varphi f\, d\gamma = F(0) \le
\|\varphi\|_1^{\sigma_2}.
$$
Therefore,  for each $\varphi\in \mathcal{A}$ such that $0\le \varphi\le 1$
we have
$$
\int_E\varphi f\, d\gamma
\le e^{2-2\sigma_2}\|\varphi\|_1^{\sigma_2}.
$$
By approximation, for every $\mu$-measurable set $A$ we obtain
$$
\int_A f\, d\mu \le e^{2-2\sigma_2}[\mu(A)]^{\sigma_2}.
$$
Taking $A=\{f\ge t\}$ with $t>0$, we get
$$
\mu(f\ge t)\le e^2\cdot t^{-\frac{1}{1-\sigma_2}},
$$
so $f\in L^p(\mu)$ for all $p<\frac{1}{1-\sigma_2}$.

We now consider the case $m>2$.
In this case, for all $a,b\in(0,+\infty)$, one has
\begin{multline*}
[-\ln F(b)]^{1/2 - 1/m} -[-\ln F(a)]^{1/2 - 1/m}
=
\tfrac{m-2}{2m}\int_a^b \frac{-F'(t)}{[-\ln F(t)]^{1/2+1/m} F(t)}\, dt
\\
\le
\tfrac{2}{m}(m-2)\lambda\int_a^b\frac{\sqrt{\theta}}{\sqrt{e^{2\theta t}-1}}\, dt
\le\tfrac{\pi}{m\sqrt\theta}(m-2)\lambda:=K.
\end{multline*}
Let $m'>2$ be such that $\frac{1}{2} - \frac{1}{m} = \frac{1}{m'}$.
Then, by convexity,
\begin{multline*}
-\ln F(b)\le\bigl(K + [-\ln F(a)]^{1/m'}\bigr)^{m'}
\\
\le
(1+\varepsilon^{-1})^{m'-1}K^{m'} + (1+\varepsilon)^{m'-1}[-\ln F(a)]
\end{multline*}
for every $\varepsilon>0$. Therefore,
$$
\ln F(a) \le \varepsilon^{1-m'}K^{m'} + (1+\varepsilon)^{1-m'}\ln F(b)
$$
and
$$
F(a) \le [F(b)]^{(1+\varepsilon)^{1-m'}}e^{\varepsilon^{1-m'}K^{m'}}.
$$
Passing to the limit as $a\to0$ and $b\to+\infty$, we obtain
$$
\int_E\varphi f\, d\gamma = F(0) \le \|\varphi\|_1^{(1+\varepsilon)^{1-m'}} e^{\varepsilon^{1-m'}K^{m'}}
$$
for all $\varepsilon>0$. Thus, for
every $\varphi\in \mathcal{A}$ such that $0\le\varphi\le1$ we have
$$
\int_E\varphi f\, d\mu \le e^{2-2(1+\varepsilon)^{1-m'}}
\|\varphi\|_1^{(1+\varepsilon)^{1-m'}} e^{\varepsilon^{1-m'}K^{m'}}.
$$
By approximation, for every $\mu$-measurable set $A$ we obtain
$$
\int_A f\, d\mu \le e^{2-2(1+\varepsilon)^{1-m'}}
\bigl[\mu(A)\bigr]^{(1+\varepsilon)^{1-m'}} e^{\varepsilon^{1-m'}K^{m'}}.
$$
We now take $A:=\{f\ge t\}$ and get
$$
\mu(f\ge t)\le e^{2} \Bigl[t^{-1}e^{\varepsilon^{1-m'}K^{m'}}\Bigr]^{\frac{1}{1-(1+\varepsilon)^{1-m'}}}.
$$
For $t>1$ we take $\varepsilon:= \bigl(\frac{2 K^{m'}}{\ln t}\bigr)^{\frac{1}{m'-1}}$,
i.e., $\varepsilon^{1-m'}K^{m'}=2^{-1}\ln t$, and get
$$
\mu(f\ge t)\le e^{2} \Bigl[e^{-2^{-1}\ln t}\Bigr]^{\frac{1}{1-(1+\varepsilon)^{1-m'}}}.
$$
Since
$$
1 - (1+\varepsilon)^{1-m'} = \int_0^\varepsilon(m'-1)(1+s)^{-m'}\, ds \le (m'-1)\varepsilon,
$$
we arrive at the estimate
$$
\mu(f\ge t)\le e^{2} \Bigl[e^{-2^{-1}\ln t}\Bigr]^{\frac{1}{(m'-1)\varepsilon}}
=
e^{2} e^{- \sigma_m[\ln t]^{\frac{m'}{m'-1}}},
$$
where
$$
\sigma_m := \frac{1}{m'-1} (2K)^{-\frac{m'}{m'-1}}=\frac{1-2/m}{1+2/m}
\Bigl(\frac{2\pi}{\sqrt{\theta}}\lambda(1-2/m)\Bigr)^{-\frac{2}{1+2/m}}.
$$
Thus, for every $s>1$ we have
$$
\mu\Bigl(e^{\varepsilon[\ln\max\{f, 1\}]^{\frac{m'}{m'-1}}} \ge s\Bigr) =
\mu\Bigl(f\ge e^{[\ln s^{1/\varepsilon}]^{\frac{m'-1}{m'}}}\Bigr)
\le e^{2} s^{- \sigma_m\varepsilon^{-1}}.
$$
Hence
$e^{\varepsilon[\ln\max\{f, 1\}]^{\frac{m'}{m'-1}}}\in L^1(\mu)$
for all $\varepsilon<\sigma_m$.

In the case of bounded $w$ we have $\|w\|_{L_{\psi_m}}\le 2^{1/m}\|w\|_\infty$
and the result follows by taking the limit as $m\to+\infty$.
The theorem is proved.
\end{proof}

\begin{remark}\label{rem1}{\rm
In the setting of the previous theorem for each $m\in (2, +\infty]$
(again $m=+\infty$ is understood as $\|w\|_\infty<\infty$)
and for each $p\in(1, +\infty)$ there is a number $C(m, p)$
such that for all $t>1$ one has
$$
\mu(f\ge t)\le  t^{-p}\exp\Bigl(2+C(m, p-1)[\theta^{-1/2}\lambda]^\frac{2}{1-2/m}\Bigr)
$$
and
$$
\int_E |f|^p\, d\mu\le 1 + \exp\Bigl(2+C(m, p)[\theta^{-1/2}\lambda]^\frac{2}{1-2/m}\Bigr)
\le 2\exp\Bigl(2+C(m, p)[\theta^{-1/2}\lambda]^\frac{2}{1-2/m}\Bigr).
$$
}\end{remark}

The abstract result above yields the following bound for solutions to Fokker--Planck--Kolmogorov
equations.

\begin{theorem}
Let $\mu=e^{-W}\, dx$ with $W\in C^2(\mathbb{R}^d)$, $D^2W\ge \theta \cdot {\rm I}$, $\theta>0$ and
let $f\in L^1(\mu)$ be a probability solution to the equation
$L_v^*[f\cdot\mu]=0$
with some $(f\cdot\mu)$-integrable vector field~$v$,
where $L_v u = \Delta u + \langle -\nabla W +v, \nabla u \rangle$.

{\rm (i)} If $v\in L_{\psi_2}(f\cdot\mu)$, then
$$
\mu(f\ge t)\le e^2 t^{-\frac{1}{1-\sigma_2}}
$$
and $f\in L^p(\mu)$ for all $p<\frac{1}{1-\sigma_2}$,
where $\sigma_2:=\exp\bigl(-\frac{2\pi}{\sqrt{\theta}}\|v\|_{L_{\psi_2}(f\cdot\mu)}\bigr)${\rm;}

{\rm (ii)} If $|v|\in L_{\psi_m}(f\cdot\mu)$ with $m>2$, then
$$
\mu(f\ge t)\le
e^{2} \exp\bigl(- \sigma_m[\ln t]^{\frac{2}{1+2/m}}\bigr) \quad \forall\, t>1
$$
and
$e^{\varepsilon[\ln \max\{f,1\}]^\frac{2}{1+2/m}}\in L^1(\mu)$ for all $\varepsilon<\sigma_m$,
where
$$
\sigma_m :=\frac{1-2/m}{1+2/m}
\Bigl(\frac{2\pi}{\sqrt{\theta}}\|v\|_{L_{\psi_2}(f\cdot\mu)}(1-2/m)\Bigr)^{-\frac{2}{1+2/m}}.
$$

{\rm (iii)} If $v$ is bounded, then
$$
\mu(f\ge t)\le
e^{2} e^{-\sigma_\infty[\ln t]^2}\quad \forall\, t>1
$$
and
$e^{\varepsilon[\ln \max\{f,1\}]^2}\in L^1(\mu)$ for all $\varepsilon<\sigma_\infty$,
where
$\sigma_\infty := \bigl(\tfrac{2\pi}{\sqrt{\theta}}\||v|\|_\infty\bigr)^{-2}$.
\end{theorem}
\begin{proof}
The previous theorem applies with $w=|v|$
and
$$
L\varphi=\Delta\varphi -\langle \nabla W,\nabla\varphi\rangle,
$$
 because  the equation reads as
$$
\int L\varphi\, f\, d\mu=-\int \langle \nabla \varphi,v\rangle \, f\, d\mu,
\quad \varphi\in C_0^\infty(\mathbb{R}^d),
$$
where $|\langle \nabla \varphi,v\rangle|\le |v|\, |\nabla \varphi|$
and $\Gamma(\varphi,\psi)=\langle \nabla \varphi, \nabla \psi\rangle$.
\end{proof}

We note that the above result is applicable to the standard Gaussian measure $\mu$
for which $\theta=1$.

Let us show that the result is nearly sharp even for the one-dimensional Ornstein--Uhlenbeck operator
and a constant $v$.

\begin{example}
{\rm
Let $\gamma$ be the standard Gaussian measure on the real line and let $\mu=f\cdot\gamma$,
where $f(x)=\exp(v x)$ with some constant $v\not=0$. Then $\mu$ satisfies the FPK-equation
with the drift $b(x)=-x+v$. We have $f\in L^p(\mu)$ for all $p>1$, but $\exp (\varepsilon f)\not\in L^1(\gamma)$
for all $\varepsilon >0$, one only has $\exp (\varepsilon |\log f|^2)\in L^1(\gamma)$
for  $\varepsilon <(2|v|^2)^{-1}$.
}\end{example}

\begin{remark}\label{rem2.5}
\rm
In the situation of the previous theorem
the operator
$$
L\varphi = \Delta \varphi + \langle -\nabla W, \nabla \varphi \rangle
$$
is symmetric on the domain $C_0^\infty$ in $L^2(\mu)$, which follows by the integration
by parts formula:
$$
\int \psi L\varphi\, d\mu=-\int \langle \nabla \psi, \nabla \varphi\rangle \, d\mu.
$$
One can introduce the divergence $\delta_\mu w$ of a locally Sobolev vector field $w$
by
$$
\delta_\mu w ={\rm div}\, w- \langle w,\nabla W\rangle .
$$
In particular, for the standard Gaussian measure $\gamma$ we have
$$
\delta_\gamma w (x) ={\rm div}\, w(x)- \langle w(x),x\rangle .
$$
For every $\varphi\in C_0^\infty$ there holds the equality
$$
\int \langle \nabla \varphi, w\rangle \, d\mu=-\int \varphi \delta_\mu w\, d\mu,
$$
which can be used as the definition of divergence in the sense of distributions
for locally integrable $w$ that is not locally Sobolev.

If $v$ is locally integrable to some power $p>d$
with respect to Lebesgue measure,
the equation $L_v^*(f\cdot \mu)=0$
for a measure $f\cdot\mu$ with a  density
$f\in L^1(\mu)$  is equivalent to the equation
\begin{equation}\label{gen-div}
Lf= \delta_\mu w
\end{equation}
with the vector field
$$
w:=f\cdot v,
$$
understood as the identity
$$
\int L\varphi f \, d\mu
=-\int\langle \nabla \varphi,v\rangle f \, d\mu \quad \forall\, \varphi\in C_0^\infty.
$$
This follows by the integration by parts formula taking into account that
$f\in W^{p,1}_{loc}$. The same is true if in the condition of the local  integrability
of $|v|^p$ Lebesgue measure is replaced by the solution $f\cdot\mu$.

It is known in the Gaussian case (see \cite[Section~4.2]{Shig})
that if $f$ and $|w|$ belong to~$L^p(\gamma)$,  $p>1$,
then $\delta_\gamma w$ and
$\delta_\gamma w -f $ belong to the negative class $W^{p,-1}(\gamma)$, i.e.,
$Lf-f\in W^{p,-1}(\gamma)$, which yields the inclusion $f\in  W^{p,1}(\gamma)$,
along with the estimate
$$
\|f\|_{p,1}\le C(p) (\| w\|_p+ \| f\|_p).
$$
where $C(p)$ depend only on~$p$.
\end{remark}

It is worth noting that the existence of a probability solution was part of our hypotheses.
Some sufficient condition on the perturbation ensuring this can be found
in~\cite{BRZ}.

Finally, we prove a result of independent interest, we owe its proof
to Stanislav Shaposhnikov.

\begin{proposition}\label{p1}
Let $\gamma$ be the standard Gaussian measure on $\mathbb{R}^d$. If a probability measure $f\cdot\gamma$
satisfies the equation $L_v^*(f\cdot \gamma)=0$, where
$v$ has compact support and is $f\cdot\gamma$-integrable, for example, is bounded, then $f$ is bounded.
\end{proposition}
\begin{proof}
Let $\varrho$ be the standard Gaussian density.
From the FPK equation we have
$$
{\rm div}\bigl(\varrho\nabla f -\varrho f v\bigr)=0
$$
and outside of some ball $Q$ we obtain
$$
Lf=0.
$$
In particular, $f$ is smooth outside of~$Q$.
Let us observe that for $\delta\in(1/2, 1)$ close to $1$ such that $1-\delta <1/d$
there exists $C_{\delta}>0$ for which
$$
f(x)\varrho(x)\le C_{\delta}e^{-(1-\delta)|x|^2/2}.
$$
Indeed,
the function
$V(x)=e^{\beta|x|^2/2}$, where
$0<\beta<1$ and $1-\delta<\beta/d$, serves as a Lyapunov function
for the operator $L_v$,  that is,
$$
L_v V(x)=\beta^2 (d+|x|^2)V(x) - \beta |x|^2 V(x)+\beta \langle x,v(x)\rangle V(x)
$$
is majorized by $- 2^{-1}\beta(1-\beta)|x|^2 V(x)$ outside of~$Q$.
On the whole space we have the estimate $L_vV\le F- V$, where $F$ is a $(f\cdot\gamma)$-integrable function
with compact support (bounded in the case where $v$ is bounded).
Therefore, $V$ is integrable against~$f\cdot\gamma$ (see \cite[Section~2.3]{BKRS}),
i.e., the function $Vf\varrho$ is integrable on~$\mathbb{R}^d$.
Thus, the function  $\Phi(x)=e^{(1-\delta)|x|^2/2}$ satisfies the hypotheses
of \cite[Theorem~3.3.1]{BKRS}. These hypotheses, in addition to some technical assumptions
about the coefficients of the equation (which are trivially fulfilled in our case)
require the inclusions $\Phi\in L^1(f\cdot\gamma)$, $|\nabla \Phi|\in L^\theta (f\cdot\gamma)$
with some $\theta>d$. These inclusions are ensured by the condition $d(1-\delta)<\beta$
along with the integrability of $V$ established above.
It follows by the cited theorem that
$$
f(x)\le C_{\delta}e^{\delta|x|^2/2}.
$$
Now let $\delta< q<1$. Outside of some ball (again denoted by $Q$) we have
$$
L\Psi(x)\le 0, \quad \Psi(x)=e^{q|x|^2/2}.
$$
Let $M=\sup_Q f$. For every $\varepsilon>0$ we have
$$
L(f-\varepsilon \Psi)\ge 0
$$
outside of $Q$. Since $f(x)-\varepsilon \Psi(x)$ tends to $-\infty$ as $|x|\to\infty$,
by the maximum principle
$$
f(x)-\varepsilon \Psi(x)\le M \quad \forall\, x.
$$
Letting $\varepsilon$ to $0$, we conclude that $f\le M$. Thus,  $f$ is bounded.
\end{proof}

Note that the assumption of compactness of the support of $v$ is important
(as follows from the example above with constant~$v$).

\subsection{$L^p$-integrable drifts}

In this subsection, we consider  probability solutions to the equation
$$
L_v^*[f\cdot\mu]=0
$$
with a vector field~$v$ belonging to $L^p(f\cdot\mu)$, where
$$
L_v u = \Delta u + \langle -\nabla W +v, \nabla u \rangle,
\quad W\in C^2(\mathbb{R}^d).
$$
It is known in the case $p = 2$ that $\sqrt{f} \in W^{2,1}(\mu)$ and the logarithmic Sobolev inequality
(which holds under $CD(\theta,\infty)$, see \cite[Proposition~5.7.1]{BGL})
yields
the bound
\begin{equation}\label{apriori_bound_p_2}
\int f \log f \,d\mu \leq \frac{1}{2\theta}\int \frac{|\nabla f|^2}{f}\,d\mu \leq \frac{1}{2\theta}\int |v|^2f\,d\mu.
\end{equation}

In the case $p = 1$ the arguments analogous to \cite{BSS19} provide
the following estimate:
$$
\int f \log^{1/4 - \varepsilon} (1 + f)\,d\mu \leq
C(\theta, \varepsilon)\bigl(1 + \|v\|_{L^1(f\cdot \mu)} \log^{1/4 - \varepsilon} (1 + \|v\|_{L^1(f\cdot \mu)} )
\bigr),
$$
where $\varepsilon > 0$.
The reasoning in \cite{BSS19} is essentially based on the a priori estimates
\begin{equation}\label{integrability_improvement_1}
\int T_{t} f \log^{1/2} (1 + T_{t}f)\,d\mu \leq Ct ^{-1/2}, \ t \in (0, 1),
\end{equation}
\begin{equation}\label{short_time_1}
\|T_t f - f\|_{L^{1}(\mu)} \leq C t^{1/2},
\end{equation}
where $C$ depends on $\theta$ and the $L^1(f\cdot\mu)$ norm of the drift $v$.
In this section, we show that for sufficiently large $p$ one
can obtain an improvement over estimate  (\ref{apriori_bound_p_2}) with respect to the integrability of
$f$.

Let us recall the dual description of the Kantorovich metric
$W_{p}(\mu_1, \mu_2)$, $p \geq 1$:
$$
\frac{1}{p}W_{p}^{p}(\mu_1, \mu_2) = \sup \biggl(
\int Q_{1}\varphi \,d\mu_1 -
\int \varphi \,d\mu_2
\biggr),
$$
where the supremum is taken over all bounded continuous functions
$\varphi$ and
$$
Q_{s}\varphi (x) := \inf_{y} \biggl(
\varphi(y) + \frac{|x - y|^{p}}{ps^{p -1}}
\biggr), \ s > 0.
$$
It is well-known that $Q_s \varphi$ satisfies the Hamilton--Jacobi equation
$$
\frac{d}{ds}Q_{s}\varphi = -\frac{1}{q}|\nabla Q_{s}\varphi|^q
$$
with initial condition $\varphi$, where
$
\frac{1}{p} + \frac{1}{q} = 1.
$

\begin{lemma}\label{le:improved_integration_by_parts}
Let $f\cdot\mu$ be a probability solution to the equation
$$
L_v^*[f\cdot\mu]=0
$$
with $v \in L^1(f\cdot \mu)$
and let $\psi$ be a bounded Lipschitz function on $\mathbb{R}_{+} \times \mathbb{R}^d$.
Then for all $u \geq t$ we have
\begin{multline*}
\int \psi(u, x) T_{u}f (x)\,d\mu -
\int \psi(t, x) T_{t}f (x)\,d\mu\\
=
\int_{[t, u]}
\int
\Bigl(
\frac{\partial\psi}{\partial s}T_{s}f -
 \langle\nabla T_{s}\psi, v\rangle f\Bigr)
 \,d\mu \,ds.
\end{multline*}
\end{lemma}
\begin{proof}
The required equality trivially holds for smooth functions $\psi$
that are finite linear combinations of functions of the form $(t, x) \mapsto \psi_1 (t) \psi_2 (x)$.
It remains to notice that the general case follows by the standard approximation arguments.
\end{proof}

\begin{theorem}\label{th:kantorovich_norm_estimates}
Let $f\cdot\mu$ be a probability solution to the equation
$$
L_v^*[f\cdot\mu]=0
$$
with $v \in L^p(f\cdot \mu)$, $p > 1$.
Then
\begin{equation}\label{w_p_ac_estimate}
W^{p}_{p}(T_{t + h}f \cdot \mu, T_{t}f \cdot \mu) \leq h^p\int |v|^p f\,d\mu,
\end{equation}
\begin{equation}\label{w_p_global_estimate}
W^{p}_{p}(f \cdot \mu, \mu) \leq
\frac{\theta^{-p + 1}}{(p -1)(q - 1)}
\int |v|^p f\,d\mu.
\end{equation}
\end{theorem}
\begin{proof}
Let  $0 \leq t < u \leq \infty$ and let $\eta$ be an increasing function on the real line such that
$\eta(t) =0, \eta(u) = 1$. The particular choice of $t, u$ and $\eta$ will be specified below.
Let us fix $\varphi \in C_{0}^{\infty}$ and set
$$
\psi(s, x) := Q_{\eta(s)}\varphi, \ s \in [u, t].
$$
Applying Lemma \ref{le:improved_integration_by_parts} we obtain the equality
\begin{multline*}
\int Q_{1}\varphi T_{u}f\,d\mu -
\int \varphi T_{t}f\,d\mu \\
=
\int_{[t, u]}\int\Bigl(
-\frac{\eta'(s)}{q}|\nabla Q_{\eta(s)} \varphi|^q T_{s}f
 - \langle
\nabla T_{s}Q_{\eta(s)}\varphi,
v\rangle f
\Bigr)
\,d\mu
\,ds.
\end{multline*}
The  gradient commutation inequality (see \cite[Theorem 3.2.4]{BGL}) provides the bound
$$
|\nabla T_{s}g| \leq e^{-\theta s}T_{s}|\nabla g|.
$$
Therefore, for each positive number $K$ that may depend on $s$ we have
\begin{multline*}
|
\langle
\nabla T_{s}Q_{\eta(s)}\varphi,
v\rangle| f \leq \frac{K}{q}|\nabla T_{s}Q_{\eta(s)}\varphi|^{q} f +
\frac{K^{-p/q}}{p}|v|^p f  \\
 \leq
\frac{K}{q}  e^{-\theta q s}T_{s} |\nabla Q_{\eta(s)}\varphi|^q f
+
\frac{K^{-p + 1}}{p}|v|^p f.
\end{multline*}
where Young's and Jensen's inequalities have been used.
Next, the numbers $t, u, \eta$ and $K$ will be picked in such a way that
the term
$$
\frac{K}{q}  e^{-\theta q s}T_{s} |\nabla Q_{\eta(s)}\varphi|^q f
$$
will be canceled by
$$
-\frac{\eta'(s)}{q}|\nabla Q_{\eta(s)} \varphi|^q T_{s}f.
$$
To establish inequality (\ref{w_p_ac_estimate}) let us set
$$
u := t + h, \ \eta(s) := \frac{s - t}{h},  \ K := 1/h.
$$
In this case
$$
\int Q_{1}\varphi T_{u}f\,d\mu -
\int \varphi T_{t}f\,d\mu
\leq
\int_{[t, t +h]}\int
\frac{h^{p - 1}}{p}|v|^p f
\,d\mu
\,ds =
\frac{h^p}{p}\int |v|^p f\, d\mu.
$$
Finally, to complete the proof of inequality (\ref{w_p_global_estimate})
we set
$$
t := 0, \ u := \infty, \ \eta(s) := 1 - e^{-\theta  s}, \ K := \theta e^{\theta(q -1)s}.
$$
In this case
\begin{multline*}
\int Q_{1}\varphi \,d\mu -
\int \varphi f\,d\mu
\leq
\int_{[0, \infty)}
\int
\frac{\theta^{-p + 1} e^{-\theta(q - 1)(p - 1)s}}{p}
|v|^p f
\,d\mu
\,ds \\
\leq
\frac{\theta^{-p + 1}}{p(q - 1)(p - 1)}
\int |v|^p f\, d\mu,
\end{multline*}
which gives our claim.
\end{proof}

The next proposition can be considered as an $L^p$-counterpart of the classical inequality
$$
\int T_t g \log T_t g\,d\mu \leq \frac{1}{4t}W^{2}_{2}(g\cdot\mu, \mu)
$$
that  is known under the $CD(0, \infty)$-condition (see \cite[inequality (11)]{BGL15}).

\begin{proposition}\label{pr:integrability_improvement}
Assume that the curvature-dimension condition $CD(0, \infty)$ holds. Then, for every probability measure $g \cdot \mu$
with finite Kantrovich distance $W_{p}(g\cdot \mu, \mu)$ of order $p \geq 1$
and every $t > 0$, we have
\begin{equation}\label{eq:improved_integrability}
\int T_{t}g \log^{p/2}(c_p + T_{t}g)\,d\mu \leq
6^{p/2 + 1} \log^{p/2}(c_p + 1) +
t^{-p/2}(3/2)^{p/2}W_{p}^{p}(g\cdot \mu, \mu),
\end{equation}
where $c_p := \max(1, e^{p/2 - 1})$.
\end{proposition}
\begin{proof}
Let us recall that the condition $CD(0, \infty)$ implies (see e.g. \cite{BGL})
Wang's Harnack inequality:
\begin{equation}\label{eq:wang_harnack}
T_{t}h^{1/2}(y) \leq \bigl(T_{t}h(x)\bigr)^{1/2} e^{|x -y|^2/4t},
\end{equation}
where $h$ is a nonnegative measurable function.
Now let us consider the following  auxiliary function:
$$
\varphi(z) := \log^{p/2}(z + c_p), \ z \geq 0.
$$
Since
$$
\varphi'(z) = \frac{p}{2}\frac{\log^{p/2 - 1}(z + c_p)}{z + c_p},
$$
$$
\varphi''(z) = \frac{p}{2}\frac{\log^{p/2 - 2}(z + c_p)
((p/2 - 1) - \log(z + c_p))
}{2(z + c_p)^2}
$$
it is easy to see that
$$
\varphi'(z) \geq 0, \ \varphi''(z) \leq 0, \ z \geq 0,
$$
i.e., $\varphi$ is increasing and concave on $\mathbb{R}_{+}$.
Now let us apply $\varphi$ to the both sides of~(\ref{eq:wang_harnack})
and take into account Jensen's inequality:
$$
T_{t} \varphi(h^{1/2})(y) \leq
\varphi(T_t h^{1/2}(y))
\leq \varphi\bigl(
\bigl(
T_{t}h(x)
\bigr)^{1/2}
e^{|x-y|^2/4t}
\bigr),
$$
$$
T_{t}\log^{p/2}(c_p + h^{1/2})(y) \leq
\log^{p/2}\bigl(c_p + (T_t h(x))^{1/2}e^{|x-y|^2/4t}\bigr)
$$
Since $c_p \geq 1$ and for $z \geq 0, z_1 \geq 1$ one has
$$
(c_p + z)^{1/2} \leq c_p + z^{1/2}, \
c_p + z^{1/2}z_1 \leq (c_p + 1)  (c_p + z)^{1/2}z_1,
$$
we obtain
$$
T_{t}\log^{p/2}(c_p + h)(y)
\leq \Bigl(
2\log (c_p + 1) + \log(c_p + T_{t}h(x)) + \frac{|x - y|^2}{2t}
\Bigr)^{p/2},
$$
$$
T_{t}\log^{p/2}(c_p + h)(y) \leq 6^{p/2} \log^{p/2}(c_p + 1) +
3^{p/2}\log^{p/2}(c_p + T_t h(x)) +
t^{-p/2}(3/2)^{p/2}|x  - y|^p.
$$
Now let us pick $h := T_t g$ and integrate this inequality
over the optimal coupling between the measures $g\cdot \mu$ and $\mu$:
$$
\int g T_{t}\log^{p/2}(c_p + T_t g)\,d\mu
\leq
6^{p/2 + 1} \log^{p/2}(c_p + 1) +
t^{-p/2}(3/2)^{p/2}W_{p}^{p}(g\cdot \mu, \mu),
$$
where Jensen's inequality has been used to obtain the bound
$$
\int \log^{p/2}(c_p + T_{2t}g)\,d\mu \leq  \log^{p/2}(c_p + 1).
$$
Taking into account the symmetry of $T_t$ we finally obtain~(\ref{eq:improved_integrability}).
\end{proof}

\begin{proposition}\label{pr:fisher}
Assume that the curvature-dimension condition $CD(0, \infty)$ holds.
Then for every nonnegative function $f$ such that $f^{1/2} \in W^{2, 1}(\mu)$ one has
$$
\int \frac{|\nabla T_t f|^2}{T_t f}\,d\mu \leq \int \frac{|\nabla f|^2}{f}\,d\mu.
$$
\end{proposition}
\begin{proof}
This easily follows by the inequality
$$
\frac{|\nabla T_t f|^2}{T_t f} \leq \frac{(T_{t} |\nabla f|)^2}{T_t f}
\leq
\frac{\Bigl(
T_{t}\bigl(|\nabla f| f^{-1/2}f^{1/2}\bigr)
\Bigr)^2}{T_t f} \leq
T_{t}\frac{|\nabla f|^2}{f}
$$
and integration with respect to $\mu$.
\end{proof}

Now we are ready to present the main theorem of this subsection.

\begin{theorem}\label{th:main_l_p}
Let us assume that the curvature condition $CD(\theta, \infty)$ holds
with some $\theta > 0$ and let $f\cdot \mu$ be a probability solution to the equation
$$
L_v^*[f\cdot\mu]=0
$$
with some vector field~$v$ belonging to
$L^p(f\cdot\mu)$, $p > 2$.
Then for every $\alpha < \min\bigl(2, \frac{p + 2}{4}\bigr)$
there exists $C> 0$ depending on $\theta, p, \alpha$ such that
$$
\int f \log^{\alpha}(1 + f)\,d\mu \leq C + C\int |v|^p f\,d\mu.
$$
\end{theorem}
\begin{proof}
Let us set
$$
\Phi_{R}(z) := \int_{[0, z]}R \wedge \frac{\log^{\alpha}(1 + x)}{\alpha}\,dx, \ z\geq 0.
$$
It is clear that
$$
\Phi'_{R}(z) = R \wedge \frac{\log^{\alpha}(1 + z)}{\alpha}, \
\Phi''_{R}(z) = I_{\{z < R'\}}\frac{\log^{\alpha - 1}(1 + z)}{1 + z}
$$
for the appropriate value $R'$.
Let us fix $\delta > 0$ and set
$$
f_{\delta, t}(x) := \frac{1}{\delta}\int_{[t, t + \delta]}T_{u}f(x)\,du,\
f_{\delta}\colon t\mapsto f_{\delta, t} \in  L^{1}(\mathbb{R}^d, W^{1, 1}[0, 1]).
$$
Taking into account Proposition \ref{pr:fisher} one can show that
$$
\int \frac{|\nabla f_{\delta, t}|^2}{f_{\delta, t}} \,d\mu
\leq \int \frac{|\nabla f|^2}{f}\,d\mu \leq  \|v\|^{2}_{L^2(f\cdot\mu)}.
$$
Indeed, similarly to the proof of Proposition \ref{pr:fisher} this follows by the chain of inequalities
\begin{multline*}
\int \frac{|\nabla f_{\delta, t}|^2}{f_{\delta, t}}\,d\mu  \leq
\int \frac{1}{f_{\delta, t}}
\Biggl[
\frac{1}{\delta}
\int_{[t, t + \delta]}|\nabla T_u f|\,du
\Biggr]^2\,d\mu \\
\leq
\int \frac{1}{f_{\delta, t}}
\Biggl[
\frac{1}{\delta}
\int_{[t, t + \delta]}|\nabla T_u f|\cdot T^{-1/2}_{u}f \cdot T^{1/2}_{u}f\,du
\Biggr]^2\,d\mu \\
\leq
\frac{1}{\delta}
\int
\int_{[t, t + \delta]}\frac{|\nabla T_u f|^2}{T_{u}f}\,du
\,d\mu
\leq \int \frac{|\nabla f|^2}{f}\,d\mu \leq \|v\|^{2}_{L^2(f\cdot\mu)}.
\end{multline*}
Inequality (\ref{w_p_ac_estimate}) from Theorem~\ref{th:kantorovich_norm_estimates} yields the bound
$$
W_{p}(T_{t}f\cdot\mu , T_{t + h}f\cdot\mu) \leq h\|v\|_{L^{p}(f\cdot\mu)}.
$$
It is easy to see that this bound implies the estimate
\begin{equation}\label{eq:lip_curve}
W_{p}(f_{\delta,t}\cdot\mu , f_{\delta,t + h}\cdot\mu) \leq h\|v\|_{L^{p}(f\cdot\mu)}.
\end{equation}
Now let us consider the curve of probability measures
$\{\mu_{\delta, t}\}_{t \in [0, 1]}$ given by
$$
t \mapsto \mu_{\delta, t} := f_{\delta, t} \cdot \mu.
$$
By the Benamou--Brenier formula (see  \cite[Theorem 8.3.1]{AGS})
there exists a time-dependent  Borel vector field $V_{\delta, t}$, such that
$$
\int |V_{\delta, t}|^p d\mu_{\delta, t} \leq \|v\|^{p}_{L^{p}(f\cdot\mu)}, \
\frac{\partial}{\partial t} \mu_{\delta, t} +
{\rm div}\bigl( V_{\delta, t} \cdot \mu_{\delta, t}\bigr) = 0.
$$
Since $\Phi_{R}$ is Lipschitz  and $f_{\delta} \in L^1(\mathbb{R}^d, W^{1, 1}[0,1])$, the mapping
$$
t \mapsto \int \Phi_{R}(f_{\delta, t})\,d\mu
$$
is absolutely continuous and
\begin{multline*}
\biggl|
\frac{d}{dt} \int \Phi_{R}(f_{\delta, t})\,d\mu
\biggr|
=
\biggl|
\int \Phi'_{R}(f_{\delta, t}) \frac{\partial}{\partial t}f_{\delta, t}\,d\mu
\biggr|
\\
\leq
\int \Phi_{R}''(f_{\delta, t})
\bigl|
\langle \nabla f_{\delta, t}, V_{\delta, t}\rangle
\bigr|
f_{\delta, t}\,d\mu
\leq
\int \frac{\log^{\alpha- 1}(1 + f_{\delta, t})}{1 + f_{\delta, t}}
\bigl|
\langle \nabla f_{\delta, t}, V_{\delta, t}\rangle
\bigr|
f_{\delta, t}\,d\mu
\\
 \leq
\biggl[
\int \log^{\frac{2p}{p - 2}(\alpha  - 1)}(1 + f_{\delta, t})f_{\delta, t}\,d\mu
\biggr]^{1/2 - 1/p}
\biggl[
\int \frac{|\nabla f_{\delta, t}|^2 }{f_{\delta, t}}\,d\mu
\biggr]^{1/2}
\biggl[
\int |V_{\delta, t} |^p f_{\delta, t}\, d\mu
\biggr]^{1/p} \\
\leq
\biggl[
\int \log^{\frac{2p}{p - 2}(\alpha  - 1)}(1 + f_{\delta, t})f_{\delta, t}\,d\mu
\biggr]^{1/2 - 1/p}
\biggl[
\int \frac{|\nabla f|^2}{f}\,d\mu
\biggr]^{1/2} \|v\|_{L^{p}(f\cdot \mu)}.
\end{multline*}
Applying Theorem \ref{th:kantorovich_norm_estimates} we obtain that there exists a constant $C > 0$ depending only on $p, \theta$ such that
$$
W_{p}^{p}(f\cdot \mu, \mu) \leq C \|v\|^{p}_{L^{p}(f\cdot\mu)}.
$$
Therefore, by Proposition \ref{pr:integrability_improvement}
$$
\int T_{t} f \log^{p/2}(1 + T_{t}f)\,d\mu \leq C t^{-p/2} \|v\|^{p}_{L^{p}(f\cdot\mu)}, \ t \in (0, 1],
$$
where the constant in the right-hand side depends only on $p$. By H\"older's inequality for all $\beta \in (0, p/2]$ we obtain
$$
\int T_{t} f \log^{\beta}(1 + T_{t}f)\,d\mu \leq
C t^{-\beta} \|v\|^{\beta}_{L^{p}(f\cdot\mu)}, \ t \in (0, 1].
$$
 Jensen's inequality  provides the bound
$$
\int f_{\delta, t} \log^{\beta}(1 + f_{\delta, t})\,d\mu \leq
C t^{-\beta} \|v\|^{\beta}_{L^{p}(f\cdot\mu)}, \ \beta \in (0, p/2],\ t \in (0, 1].
$$
Due to the assumptions $\alpha < (p + 2)/4, \ 2 < p$ we have
$$
\frac{2p}{p - 2}(\alpha - 1) \leq \frac{p}{2}, \
\alpha + 1 \leq p.
$$
Consequently,
$$
\biggl[
\int \log^{\frac{2p}{p - 2}(\alpha - 1)}(1 + T_t f)T_{t}f\,d\mu
\biggr]^{1/2 - 1/p}
\leq
C t^{-\alpha + 1} \|v\|^{\alpha - 1}_{L^{p}(f\cdot\mu)}.
$$
Combining the established estimates we obtain
$$
\biggl|\frac{d}{dt} \int \Phi_{R}(f_{\delta, t})\,d\mu\biggr| \leq
Ct^{-\alpha + 1} \|v\|^{\alpha - 1}_{L^p(f\cdot\mu)} \|v\|_{L^2(f\cdot\mu)}
\|v\|_{L^p(f\cdot\mu)}  \leq
Ct^{-\alpha + 1} \|v\|^{\alpha + 1}_{L^p(f\cdot\mu)}
$$
and
\begin{multline*}
\int \Phi_{R}(f_{\delta, t})\,d\mu \leq
 \biggl|
\int_{[t, 1]} \frac{d}{du} \int\Phi_{R}(f_{\delta, u})\,d\mu \,du
\biggr| +
\int \Phi_{R}(f_{\delta, 1})\,d\mu \\
\leq
C(\alpha, p, \theta) \Bigl[
1 +
\|v\|^{p}_{L^{p}(f\cdot \mu)}
\Bigr],
\end{multline*}
where the assumption $\alpha < 2$ has been used.
Since the obtained bound does not depend on
$\delta, t, R$, applying Fatou's lemma now it is easy to complete the proof.
\end{proof}

\section{Integrability of gradients}

In this section we consider the case of the standard Gaussian measure $\gamma$ on $\mathbb{R}^d$
and its infinite-dimensional analog and prove that the density $f$ of the perturbed equation
with respect to $\gamma$ belongs to the Sobolev space $W^{p,1}(\gamma)$.
It has already been noted in Remark~\ref{rem2.5}  that if $|fv|\in L^p(\gamma)$ with some $p>1$, then
$f\in W^{p,1}(\gamma)$. So it is necessary to study the integrability of~$fv$.
For example, if $f$ belongs to all $L^r(\gamma)$, as it holds under the appropriate assumptions
in Section~2,
 and $|v|\in L^p(\gamma)$ or $|v|\in L^p(f\cdot \gamma)$ for some $p>1$,
 then we obtain the inclusion $f\in W^{p-\varepsilon, 1}(\gamma)$ for each~$\varepsilon>0$.
 In the next theorem we use Orlicz norms to obtain sufficient conditions for
 the inclusion $f\in W^{p, 1}(\gamma)$ in terms of integrability of~$|v|$.

Let us recall the Poincar\'e inequality
$$
\int |f-I(f)|^p\, d\gamma\le C(p) \int |\nabla f|^p\, d\gamma, \quad f\in W^{p,1}(\gamma),
$$
where $I(f)$ is the integral of $f$.

There is also the $L^p$-version of the logarithmic Sobolev inequality
$$
\int |f|^p \log |f|\, d\gamma
\le \frac{p}{2} \|\nabla f\|_p^2 \|f\|_p^{p-2} +  \|f\|_{p}^{p}\log \|f\|_p,
$$
which follows from the standard logarithmic Sobolev inequality
$$
\int |f|^2 \log |f|\, d\gamma
\le \|\nabla f\|_2^2  + \|f\|_{2}^{2}\log \|f\|_2
$$
by considering $|f|^{p/2}$ in place of~$f$.

It follows from these inequalities
that for every $\varepsilon>0$ there is a number $C(\varepsilon,p)$
such that
\begin{equation}\label{poin}
\|f\|_p\le \varepsilon \|\nabla f\|_p+ C(\varepsilon,p) \|f\|_1
\quad \forall\, f\in W^{p,1}(\gamma).
\end{equation}
Indeed, suppose first that $f$ has zero intergal. If the claim is false,
we can find functions $f_n\in W^{p,1}(\gamma)$ with zero integrals such that
$\|\nabla f_n\|_p=1$ and
$$
\|f_n\|_p\ge \varepsilon +n \|f_n\|_1.
$$
By the Poincar\'e inequality $\|f_n\|_p\le C(p)$. Hence
$\|f_n\|_1\to 0$. By the logarithmic Sobolev inequality the integrals
of the functions $|f_n|^p\log (1+|f_n|)$ are uniformly bounded,
so $\|f_n\|_p\to 0$, which is  a contradiction.
Hence the desired constant exists for functions with zero integrals.
In the general case we obtain
$$
\|f-I(f)\|_p\le \varepsilon \|\nabla f\|_p+ C \|f-I(f)\|_1.
$$
 Hence for $f$ we obtain a similar bound
with $2C+1$ in place of~$C$.

Suppose that $\mu$ is a probability measure on $\mathbb{R}^d$
satisfying
the Fokker--Planck--Kolmogorov equation
$$
L_v^*\mu=0
$$
with
$$
L_v\varphi (x)=\Delta \varphi(x)+\langle -x+v(x),\nabla\varphi(x)\rangle,
$$
where $v$ is a Borel vector field such that $|v|\in L^1(\mu)$.
Then  $\mu=f\cdot\gamma$.
We already know that if $|v|$ is sufficiently integrable, then $f$ is integrable
to all powers and even better. The condition $|v|\in L^1(\mu)$ is not enough
for the inclusion $|\nabla f|\in L^1(\gamma)$.
The next result gives sufficient conditions
for integrability of  $|\nabla f|$ to high powers.
If $v$ is bounded, then we can use (\ref{poin}) to get the bound
$$
\|f\|_{p,1}\le C(p) \|\, |v|f\,\|_p\le C(p)\|v\|_\infty
(\varepsilon \|\nabla f\|_p+ C(\varepsilon,p)),
$$
which after taking $\varepsilon=C(p)^{-1}/2$ leads to
$$
\|f\|_{p,1}\le C'(p) \|v\|_\infty +C'(p) \quad \hbox{if } \|v\|_\infty\le 1.
$$
For $\|v\|_\infty\ge 1$ this gives a nonlinear bound
$\|f\|_{p,1}\le C(p, \|v\|_\infty)$, but a more constructive estimate is obtained below.
For unbounded $v$ we estimate $|v|f$ by means of suitable Orlicz norms.

\begin{theorem}\label{t-grad}
Let $\mu=f\cdot\gamma$ be a  probability solution to the equation
$L_v^*\mu=0$ with a vector field $v$ such that
$$
|v|\in L_{\psi_m}(f\cdot\gamma)
\quad \hbox{for some } m\in (2, +\infty],
$$
where $m=+\infty$ is understood as $\|v\|_\infty<\infty$.
Then $f\in W^{p, 1}(\gamma)$ for every $p>1$ and for any such $p$
there are numbers $C_1:=C_1(p, m)$ and $C_2:=C_2(p, m)$,
depending only on $m$ and $p$, such that
$$
\|\nabla f\|_{L^p(\gamma)}\le
C_1 \|\,|v|\,\|_{L_{\psi_m}(f\cdot \gamma)}
\exp\Bigl(C_2\|\,|v|\,\|_{L_{\psi_m}(f\cdot\gamma)}^{\frac{2}{1-2/m}}\Bigr).
$$
If $|v|\in L_{\psi_2}(f\cdot\gamma)$, then
$f\in W^{p, 1}(\gamma)$ for every $p\in(1, p^*)$, where
$$p^* = \bigl(1-e^{-2\pi\|\,|v|\,\|_{L_{\psi_2}(f\cdot\gamma)}}\bigr)^{-1}.$$
\end{theorem}
\begin{proof}
As explained above, the integrability of $f$ established in the previous section
implies the inclusion to the Sobolev classes. We now study bounds on the norms of $|\nabla f|$.
If $m>2$, we note that
$$
\|\, |v|f\, \|_p \le \biggl(\int |v|^{2p}f\, d\gamma\biggr)^{1/(2p)}
\biggl(\int f^{2p-1}\, d\gamma\biggr)^{1/(2p)}.
$$
The first term is bounded by $2 (2p)^{1/m}\|\,|v|\,\|_{L_{\psi_m}(f\cdot\gamma)}$
and the second term is estimated according to Remark~\ref{rem1}.

If $m=2$, then $p\in(1, p^*)$ and for every $q\in (1, \frac{p^*-1}{p-1})$
we have
$$
\| \,|v|f\, \|_p \le \biggl(\int |v|^{pq'}f\, d\gamma\biggr)^{1/(pq')}
\biggl(\int f^{(p-1)q+1}\, d\gamma\biggr)^{1/(pq)}.
$$
The first term again is bounded by $2 (pq')^{1/2}\|\,|v|\,\|_{L_{\psi_2}(f\cdot\gamma)}$
and the second one by Theorem~\ref{main1} is bounded by
$$
\bigl(1 + e^2(p^*-1 - q(p-1))^{-1}\bigr)^{1/(pq)}.
$$
Thus, $\|\nabla f\|_p<\infty$ for each $p\in(1, p^*)$.
\end{proof}

The constants above are independent of the dimension $d$, so our finite-dimensional estimate
extends to the infinite-dimensional case as follows.

The most transparent way of formulating an infinite-dimensional analog
is to use the standard Gaussian measure $\gamma$ on the space $\mathbb{R}^\infty$
of all real sequences (the countable power of the real line) equipped with its natural
Borel $\sigma$-algebra generated by the coordinated functions. This measure $\gamma$ is just
the countable power of the standard Gaussian measure on the real line.
The Cameron--Martin space of $\gamma$ is the usual Hilbert space $l^2$ with its
natural norm $|h|_H=\Bigl(\sum_{n=1}^\infty h_n^2\Bigr)^{1/2}$ and
the corresponding inner product $(h,k)_H$.

The Ornstein--Uhlenbeck operator $L$
is first defined on the space $\mathcal{F}\mathcal{C}$ of cylindrical functions of the form
$$
\varphi(x)=\varphi_0(x_1,\ldots,x_n), \quad \varphi_0\in C_b^\infty(\mathbb{R}^n)
$$
by the finite-dimensional expression
$$
L\varphi(x)=
\sum_{i=1}^n [\partial_{x_i}^2\varphi(x)-x_i\partial_{x_i}\varphi(x)].
$$
The Sobolev norms on such functions are defined
by
$$
\|\varphi\|_{p,1}=\biggl(\int |\varphi|^p\, d\gamma \biggr)^{1/p}
+\biggl(\int |\nabla_H \varphi|^p\, d\gamma \biggr)^{1/p},
$$
where $\nabla_H \varphi=(\partial_{x_i}\varphi)\in H$.
The Sobolev space $W^{p,1}(\gamma)$ is the completion of
$\mathcal{F}\mathcal{C}$ with respect to this norm.

There is a smaller convenient
subclass in $\mathcal{F}\mathcal{C}$: the set
$\mathcal{F}\mathcal{C}_0$ of functions $\varphi$ for which the corresponding
function $\varphi_0$ can be taken with compact support. This set is not a linear
subspace, since a function of one variable
as a function of two variables has no compact support.
Nevertheless, $W^{p,1}(\gamma)$ equals the completion of this subset
with respect to the metric generated by the Sobolev norm.

Given a Borel vector field $v=(v_i)$ with values in $H$, we
introduce the perturbed operator
$$
L_v \varphi =L\varphi+ (v,\nabla_H \varphi)_H, \quad \varphi \in \mathcal{F}\mathcal{C},
$$
and obtain the
corresponding Fokker--Planck--Kolmogorov equation
$$
L_v^*\mu=0
$$
with respect to Borel probability measures $\mu$ on $\mathbb{R}^\infty$ such  that
$v_i\in L^1(\mu)$, understood as the identity
\begin{equation}\label{ie1}
\int L_v \varphi\, d\mu=0 \quad \forall\, \varphi\in \mathcal{F}\mathcal{C}_0.
\end{equation}

It is also possible to introduce a stronger form of this equation requiring the last
identity for all $\varphi\in \mathcal{F}\mathcal{C}$, but for this we need in addition the
integrability of $x_i$ and $v_i$ against~$\mu$.
The integrability of the coordinate functions becomes important even if $v$ is bounded.
An advantage of dealing with the nonlinear class $\mathcal{F}\mathcal{C}_0$ of test
functions is that the equation is meaningful if $v_i$ are bounded.

Assuming the integrability of $v_i$, it
is readily seen that $\mu$ satisfies the FPK precisely when its finite-dimensional
projections $\mu_n$ satisfy the equations on $\mathbb{R}^n$ with the drifts obtained by
perturbations of $-x$ by the fields $v^n=(E_nv_1,\ldots, E_nv_n)$, where
$E_n v_j$ is the conditional expectation of $v_j$ with respect to the projection on $\mathbb{R}^n$
and the measure~$\mu$. In particular, if $|v|_H\le C$, then also $|v^n|\le C$, and if
$|v|_H\in L^2(\mu)$, then $|v^n|\in L^2(\mu_n)$.

From the finite-dimensional result we obtain the following conclusion.

\begin{corollary}
If $\mu$ be a  probability measure satisfying equation {\rm(\ref{ie1})}
with $|v|_H\in L^1(\mu)$. Then $\mu=f\cdot\gamma$ and the following assertions are true.

{\rm (i)} If $v\in L_{\psi_2}(\mu)$, then
$$
\gamma(f\ge t)\le e^2 t^{-\frac{1}{1-\sigma_2}}
$$
and $f\in L^p(\gamma)$ for all $p<\frac{1}{1-\sigma_2}$,
where $\sigma_2:=\exp\bigl(-2\pi \|v\|_{L_{\psi_2}(\mu)}\bigr)${\rm;}

{\rm (ii)} If $|v|_H\in L_{\psi_m}(\mu)$ with $m>2$, then
$$
\gamma(f\ge t)\le
e^{2} \exp\bigl(- \sigma_m[\ln t]^{\frac{2}{1+2/m}}\bigr) \quad \forall \, t>1
$$
and
$e^{\varepsilon[\ln \max\{f,1\}]^\frac{2}{1+2/m}}\in L^1(\gamma)$ for all $\varepsilon<\sigma_m$,
where
$$
\sigma_m :=\frac{1-2/m}{1+2/m}
\Bigl(2\pi \|v\|_{L_{\psi_2}(f\cdot\mu)}(1-2/m)\Bigr)^{-\frac{2}{1+2/m}}.
$$

{\rm (iii)} If $|v|_H$ is bounded, then
$$
\gamma(f\ge t)\le
e^{2} e^{-\sigma_\infty[\ln t]^2}\quad \forall\, t>1
$$
and
$e^{\varepsilon[\ln \max\{f,1\}]^2}\in L^1(\mu)$ for all $\varepsilon<\sigma_\infty$,
where
$\sigma_\infty := (2\pi\|\,|v|\,\|_\infty)^{-2}$.
\end{corollary}
\begin{proof}
The measures $\mu_n$ are given by densities $f_n$
with respect to the standard Gaussian measures $\gamma_n$ on $\mathbb{R}^n$.
The sequence $\{f_n\}$ is  a martingale
with respect to the Gaussian measure $\gamma$ and the sequence of $\sigma$-algebras generated by the projections
to~$\mathbb{R}^n$. According to \cite{BSS19}, this sequence
is uniformly integrable, hence  converges in  $L^1(\gamma)$ to some function $f\in L^1(\gamma)$.
It is readily seen that $\mu=f\cdot\gamma$. Convergence also holds in all $L^p(\gamma)$.
Moreover, by Jensen's inequality for conditional expectations there hold
uniform bounds on the Orlicz norms of $|v_n|_H$, which imply the corresponding
bounds for $f_n$ and consequently for~$f$.
\end{proof}

\begin{corollary}
If $\mu$ is a  probability measure satisfying the equation {\rm(\ref{ie1})}
and the hypotheses of Theorem~{\rm\ref{t-grad}} are fulfilled
with $|v|_H$ in place of $|v|$, then the conclusion of that corollary
is true.
\end{corollary}
\begin{proof}
In the proof of the previous corollary we have
$f_n\in W^{p,1}(\gamma_n)$ and there hold the stated
bounds on $\nabla f_n$.
By the known properties of Sobolev spaces the same bounds hold for~$\nabla f$.

It is worth mentioning that it is not necessary
to refer to the  finite-dimensional case, because the reasoning applied in the previous
section remains in force in the infinite-dimensional case once we have
the integrability of $f$ used there.
\end{proof}

Note again that the inclusion $f\in W^{2,1}(\gamma)$ follows from \cite{Shig87} and~\cite{BR95}
and the inclusions $f\in L^p(\gamma)$ with $p$ from some interval follow from~\cite{Hino98}. If
$|v|_H\in L^2(\mu)$, then $\mu=f\cdot\gamma$ with $\sqrt{f}\in W^{2,1}(\gamma)$ according
to~\cite{BR95}, and if $|v|_H\in L^1(\mu)$, then $f$ exists, but can fail to be in $W^{1,1}(\gamma)$.

In the case of an abstract centered Radon Gaussian measure $\gamma$ on a locally convex space~$X$,
having the Cameron--Martin subspace~$H$ (the subspace of vectors with
finite Cameron--Martin norm $|h|_H=\sup \{l(h)\colon l\in X^*, \, \|l\|_{L^2(\gamma)}\le 1\}$),
the same conclusion holds with the following change in the formulation:
in place of $v_i$ we consider the functions $l_i(v)$, where $\{l_i\}$ is an orthonormal
base in the dual space $X^*$ considered as a subspace in $L^2(\gamma)$ (it is known
that such a basis exists, see, e.g., \cite{B98}). The proof is the same, but it is not necessary
to repeat the proof, using instead the following fact (Tsirelson's theorem, see~\cite{B98}):
if $\gamma$ is not concentrated on a finite-dimensional subspace, then
the mapping $T\colon x\mapsto (l_i(x))$ from $X$ to $\mathbb{R}^\infty$
takes $\gamma$ to the standard Gaussian measure on $\mathbb{R}^\infty$
and it is a Borel isomorphism between two Borel linear subspaces of full measure,
in addition, its restriction is an isometry of the Cameron--Martin subspaces.

It is worth noting that the assertion from Proposition~\ref{p1} about bounded densities
for compactly supported perturbations
does not extend to the infinite-dimensional case. Indeed, let us take
functions $v_n$ on the real line such that $v_n(t)=2^{-n}$ if $|t|\le T_n$, where
$T_n$ will be large enough. If $|t|>T_n$, we set $v_n(t)=0$.
Let $f_n$ be the density of the solution to the equation with the drift $-t+v_n(t)$
with respect to the standard Gaussian measure with density $\varrho$. Then
$$
f_n(t)=\exp\biggl(\int_0^t v_n(s)\, ds+c_n\biggr),
$$
where $c_n$ is the normalization constant. On $[-T_n,T_n]$ we have
$$
f_n(t)=\exp(t2^{-n}+c_n).
$$
Take $T_n>4^n$ so large that the integral of $\exp(2^{-n} t)\varrho(t)$ over $[-T_n,T_n]$
is between $\exp(2^{-2n-1}-1)$ and $\exp(2^{-2n-1}+1)$,
which is possible, since
the integral of $\exp(2^{-n} t)\varrho(t)$ over $\mathbb{R}$ is $\exp(2^{-2n-1})$.
Then $c_n\ge  -2$, hence $f_n(4^n)\ge 2^n-2$. Taking $v_n(x)=v_n(x_n)$
on $\mathbb{R}^\infty$, we obtain a vector field with $|v|_H\le 1$ and compact
support in $\mathbb{R}^\infty$, for which the corresponding probability
solution has an unbounded density with respect to the standard Gaussian measure
(it equals $\prod_{n=1}^\infty f_n(x_n)$). Both measures can be also
regarded on the weighted Hilbert space of sequences with $\sum_{n=1}^\infty
(2T_n)^2 x_n^2<\infty$, in which $v$ also has compact support.

\begin{remark}
\rm
The same reasoning applies to more general measures on $\mathbb{R}^\infty$ in place of~$\gamma$,
namely, to any uniformly log-concave measure $\mu$, that is, a probability measure whose
projections on the spaces $\mathbb{R}^n$ have densities $e^{-W_n}$ with convex functions
$W_n$ such that $D^2W_n\ge \theta \cdot {\rm I}$ with a common constant $\theta>0$.
\end{remark}

{\bf Acknowledgements.}
This research is supported by the Russian Science Foundation Grant 17-11-01058 at
Lomonosov Moscow State University (the results in Section 2.1 and Section~3).
The second author is a winner of the ``Young Russian Mathematics'' contest and thanks its sponsors and jury.
The work of A.V. Shaposhnikov (the results in Section~2.2) was performed at the Steklov International Mathematical Center and supported
by the Ministry of Science and Higher Education of the Russian Federation (agreement no. 075-15-2019-1614).

\end{document}